\documentclass{article}
\usepackage{amsfonts}
\usepackage{amsmath}
\usepackage{enumerate}
\usepackage{amsthm}
\usepackage{hyperref}
\usepackage[top=1.2 in, bottom=1.2 in, left=1.2 in, right=1.2 in]{geometry}

\title{Schwarz type lemmas for generalized holomorphic maps between pseudo-Hermitian manifolds and Hermitian manifolds}
\author{Tian Chong, Yuxin Dong\footnote{Supported by NSFC grant No. 11771087, and LMNS, Fudan}, Yibin Ren\footnote{Supported by NSFC grant No. 11801517}, Weike Yu\footnote{The corresponding author}}
\begin{document}

\newtheorem{definition}{Definition}[section]
\newtheorem{theorem}[definition]{Theorem}
\newtheorem{proposition}[definition]{Proposition}
\newtheorem{lemma}[definition]{Lemma}
\newtheorem{corollary}[definition]{Corollary}
\newtheorem{remark}[definition]{Remark}
\newtheorem{example}[definition]{Example}
\numberwithin{equation}{section}

\maketitle
\begin{abstract}
In this paper, we consider some generalized holomorphic maps between pseudo-Hermitian manifolds and Hermitian manifolds. By Bochner formulas and comparison theorems, we establish related Schwarz type results. As corollaries, Liouville theorem and little Picard theorem for basic CR functions are deduced. Finally, we study CR Carath\'eodory pseudodistance on CR manifolds. 
\end{abstract}
\section{Introduction}
\label{intro}
The classical Schwarz-Pick lemma \cite{Pic15} states that every holomorphic map from the unit disc $D$ of $\mathbb{C}$ into itself is distance-decreasing with respect to the Poincar\'e metric. This was later generalized by Ahlfors \cite{Ahl38} to holomorphic maps from the unit disc $D$ into a Riemannian surface with curvature bounded above by a negative constant. This lemma plays an important role in complex analysis and differential geometry, and has been extended to holomorphic maps between higher dimensional complex manifolds (cf. \cite{CCL79}, \cite{Che68}, \cite{Lu68}, \cite{Roy80}, \cite{Yau78}, etc.) and harmonic maps between Riemannian manifolds (cf. \cite{GH77}, \cite{She84}, etc.). An extremely useful generalization is Yau's Schwarz lemma \cite{Yau78}, which says that every holomorphic map from a complete K\"ahler manifold with Ricci curvature bounded from below by a constant $-K_1\leq 0$ into a Hermitian manifold with holomorphic bisectional curvature bounded from above by a constant $-K_2<0$ is distance-decreasing up to a constant $K_1 /K_2$. Later, Tosatti \cite{Tos07} generalized Yau's result to the Hermitian case. Recently, the authors in \cite{DRY19} extended this lemma to generalized holomorphic maps between pseudo-Hermitian manifolds.
 
The main purpose of this paper is to generalize Yau's Schwarz lemma to two classes of generalized holomorphic maps between pseudo-Hermitian manifolds and Hermitian manifolds, which are called $(J, J^N)$-holomorphic maps (see Definition \ref{definition 3}) and $(J^N,J)$-holomorphic maps (see Definition \ref{definition 4}) respectively. By computing the Bochner formulas for these maps and using the maximum principle, we derive the following Schwarz type lemmas. 
\begin{theorem}\label{theorem 1.1}
Let $(M^{2m+1},HM,J,\theta)$ be a complete pseudo-Hermitian manifold with pseudo-Hermitian Ricci curvature bounded from below by $-K_1\leq0$ and $\|A\|_{C^1}$ bounded from above where $A$ is the pseudo-Hermitian torsion. Let $(N^n,J^N,h)$ be a Hermitian manifold with holomorphic bisectional curvature bounded from above by $-K_2<0$. Then for any $(J,J^N)$-holomorphic map $f: M\rightarrow N$, we have
\[f^*h\leq\frac{K_1}{K_2}G_\theta.\]
In particular, if $K_1=0$, every $(J,J^N)$-holomorphic map from $M$ into $N$ is constant.
\end{theorem}
\begin{theorem}\label{theorem 1.2}
 Let $(N^n,J^N,h)$ be a complete K\"ahler manifold with Ricci curvature bounded from below by $-K_1\leq0$. Let $(M^{2m+1},HM,J,\theta)$ be a Sasakian manifold with pseudo-Hermitian bisectional curvature bounded from above by $-K_2<0$.  Then for any $(J^N, J)$-holomorphic map $g: N\rightarrow M$, we have
\[ g^*G_\theta \leq\frac{K_1}{K_2}h.\]
In particular, if $K_1=0$, any $(J^N, J)$-holomorphic map is horizontally constant.
\end{theorem}
We remark that from Theorem \ref{theorem 1.1}, we can deduce Liouville theorem and little Picard theorem for basic CR functions (see Corollary \ref{corollary 4.2} and \ref{corollary 4.3} for details). In Theorem \ref{theorem 1.2}, when $\dim_{\mathbb{C}} N=1$, the hypothesis of pseudo-Hermitian bisectional curvature on $M$ can be replaced by pseudo-Hermitian sectional curvature(see Corollary \ref{corollary 5.3}).

As an application of Schwarz lemmas, we introduce CR Carath\'eodory pseudodistance on CR manifolds, which is invariant under CR isomorphisms. Making use of the relationships between Carath\'eodory pseudodistance and CR Carath\'eodory pseudodistance as well as the Schwarz lemma, we can give anorher Liouville theorem for $(J,J^N)$-holomorphic maps.

\section{Preliminaries}\label{sec2}
In this section, we will present some notations and facts of pseudo-Hermitian geometry and Hermitian geometry (cf. \cite{BG08}, \cite{DT07} for details).
\begin{definition}\label{definition 1}
Let $M$ be a real $2m+1$ dimensional orientable $C^\infty$ manifold. A CR structure on $M$ is a complex subbundle $T_{1,0}M$ of complex rank m of the complexified tangent bundle $TM\otimes\mathbb{C}$ satisfying
\begin{enumerate}[(i)]
\item $T_{1,0}M\cap T_{0,1}M=\{0\}, T_{0,1}M=\overline{T_{1,0}M}$
\item $[\Gamma(T_{1,0}M),\Gamma(T_{1,0}M)]\subseteq\Gamma(T_{1,0}M)$ \label{definition 1 (ii)}
\end{enumerate}
Then the pair $(M,T_{1,0}M)$ is called a CR manifold.
\end{definition}
The complex subbundle $T_{1,0}$ corresponds to a real subbundle of $TM$:
\begin{align}
HM=Re\{T_{1,0}M\oplus T_{0,1}M\}
\end{align}
which is called Levi distribution and carries a natural complex structure $J$ defined by $J(X+\bar{X})=i(X-\bar{X})$ for any $X\in T_{1,0}M$. The CR structure can also be said to be $(HM,J)$.  Let $(M, HM, J)$ and $(\tilde{M}, H\tilde{M}, \tilde{J})$ be two CR manifolds. A smooth map $f: M\rightarrow \tilde{M}$ is a CR map if it preserves the CR structures, that is, $df(HM)\subseteq H\tilde{M}$ and $df\circ J= \tilde{J}\circ df$ on $HM$. Moreover, if $f$ is a $C^\infty$ diffeomorphism, it is referred to as a CR isomorphism. 

Since $M$ is orientable and $HM$ is oriented by its complex structure $J$, it follows that there exist a global nowhere vanishing 1-form $\theta$ such that $HM=ker(\theta)$. Such $\theta$ is called a pseudo-Hermitian structure on $M$. The Levi form $L_\theta$ of a given pseudo-Hermitian structure $\theta$ is defined by
\begin{align}
L_\theta(X,Y)=d\theta(X,JY)
\end{align}
for any $X,Y\in HM$. The integrability condition \ref{definition 1 (ii)} in Definition \ref{definition 1} implies $L_\theta$ is $J$-invariant and symmetric. If $L_\theta$ is positive definite, then $(M,HM,J)$ is said to be strictly pseudoconvex. The quadruple $(M,HM,J,\theta)$ is called a pseudo-Hermitian manifold.

Since $L_\theta$ is positive definite, there is a step-2 sub-Riemannian structure $( HM, L_\theta)$ on $M$, and all sections of $HM$ together with their Lie brackets span $T_xM$ at each point $x\in M$. We say that a Lipschitz curve $\gamma: [0, l]\rightarrow M$ is horizontal if $\gamma'(t)\in HM$ a.e. $t\in[0, l]$. For any two points $p, q\in M$, by the theorem of Chow-Rashevsky (cf. \cite{Cho02}, \cite{Ras38}), there always exist such horizontal curves joining $p$ and $q$. Consequently, we may define the so-called Carnot-Carath\'eodory distance:
\begin{align*}
d^M_{cc}(p, q)=\inf\{\int_0^l \sqrt{L_\theta(\gamma', \gamma')}\, dt |\ \  &\gamma: [0, l]\rightarrow M \ \text{is a horizontal curve with}\\
&\ \gamma(0)=p \ \text{and} \  \gamma(l)=q  \}
\end{align*}
which induces to a metric space structure on $(M, HM, L_\theta)$.

For a pseudo-Hermitian manifold $(M,HM,J,\theta)$, there is a unique globally defined nowhere zero tangent vector field $\xi$ on $M$ called Reeb vector field such that $\theta(\xi)=1$, $\xi\rfloor d\theta=0$. Consequently there is a splitting of the tangent bundle $TM$
\begin{align}
TM=HM\oplus L,
\end{align}
where $L$ is the trivial line bundle generated by $\xi$.

Let $\pi_H:TM\rightarrow HM$ denote the natural projection morphism. Set
\begin{align}
G_\theta(X,Y)=L_\theta(\pi_HX, \pi_HY)
\end{align}
for any $X, Y\in TM$. We extend $J$ to an endomorphism of $TM$ by requiring that
\begin{align}
J\xi=0.
\end{align}
The $J$-invariance of $L_\theta$ implies $G_\theta$ is also $J$-invariant. Since $G_\theta$ coincides with $L_\theta$ on $HM\times HM$, Levi form $L_\theta$ can be extended to a Riemannian metric $g_\theta$ on $M$ by
\begin{align}
g_\theta=G_\theta+\theta\otimes\theta.
\end{align}
This metric is usually called the Webster metric. On a pseudo-Hermitian manifold, there is a canonical connection preserving the CR structure and the Webster metric, which is usually called the Tanaka-Webster connection.
\begin{theorem}[(cf. \cite{DT07})]\label{theorem 2.1}
Let $(M,HM,J,\theta)$ be a pseudo-Hermitian manifold with the Reeb vector field $\xi$ and Webster metric $g_\theta$. Then there is a unique linear connection $\nabla$ on $M$ satisfying the following axioms:
\begin{enumerate}[(i)]
\item $HM$ is parallel with respect to $\nabla$.
\item $\nabla J=0$, $\nabla g_\theta=0$.
\item The torsion $T_\nabla$ of $\nabla$ satisfies
\begin{align*}
T_\nabla(X,Y)=2d\theta(X,Y)\xi\ and\  T_\nabla(\xi,JX)+JT_\nabla(\xi,X)=0
\end{align*}
for any $X,Y\in HM$.
\end{enumerate}
\end{theorem}
The pseudo-Hermitian torsion, denoted by $\tau$, is a $TM$-valued $1$-form defined by $\tau(X)=T_\nabla(\xi,X)$ for any $X\in TM$. Set
\begin{align}
A(X,Y)=g_\theta(\tau(X),Y)
\end{align}
for any $X,Y\in TM$. A pseudo-Hermitian manifold is called a Sasakian manifold if $\tau\equiv0$. Note that the properties of $\nabla$ in Theorem \ref{theorem 2.1} imply that
$\tau(T_{1,0}M)\subset T_{0,1}M$ and $A$ is a trace-free symmetric tensor field.

Suppose that $(M^{2m+1},HM,J,\theta)$ is a real $2m+1$ dimensional pseudo-Hermitian manifold with the Webster metric $g_{\theta}$. Let $(e_i)^m_{i=1}$ be a unitary frame of $T_{1,0}M$ with respect to $g_{\theta}$ and $(\theta^i)^m_{i=1}$ be its dual frame. Then the "horizontal component of  $g_\theta$" may be expressed as
\begin{align}
G_\theta=\sum_{i=1}^m\theta^i\theta^{\bar{i}}.
\end{align}
From \cite{W+78}, we know the following structure equations for the Tanaka-Webster connection:
\begin{gather}
d\theta=2\sqrt{-1}\sum_j\theta^j\wedge\theta^{\bar{j}}\\
d\theta^i=\sum_j\theta^j\wedge\theta^i_j+\theta\wedge\tau^i\label{2.10}\\
d\theta^i_j=\sum_k\theta^k_j\wedge\theta^i_k+\Pi^i_j\label{2.11}\\
\theta^i_j+\theta^{\bar{j}}_{\bar{i}}=0
\end{gather}
where $\theta^i_j$ is the Tanaka-Webster connection 1-form with respect to $(e_i)_{i=1}^n$, $\tau^i=\sum_jA^i_{\bar{j}}\theta^{\bar{j}}=\sum_j g_\theta(\tau e_i, e_j)\theta^{\bar{j}}$, and
\begin{align}
\Pi^i_j=2\sqrt{-1}&\theta^i\wedge\tau^{\bar{j}}-2\sqrt{-1}\tau^i\wedge\theta^{\bar{j}}+\sum_{k,l}R^i_{jk\bar{l}}\theta^k\wedge\theta^{\bar{l}}\notag\\
                  &+\sum_k(W^i_{jk}\theta^k\wedge\theta-W^i_{j\bar{k}}\theta^{\bar{k}}\wedge\theta )\end{align}
where $W^i_{jk}=A^{\bar{k}}_{j,\bar{i}}$, $W^i_{j\bar{k}}=A^i_{\bar{k},j}$, and $R^i_{jk\bar{l}}$ are the components of curvature tensor with respect to Tanaka-Webster connection. Set $R_{i\bar{j}k\bar{l}}=R^j_{ik\bar{l}}$, then we know that
\begin{align}
R_{i\bar{j}k\bar{l}}&=-R_{\bar{j}ik\bar{l}}=-R_{i\bar{j}\bar{l}k}\notag\\
R_{i\bar{j}k\bar{l}}&=R_{k\bar{j}i\bar{l}}=R_{k\bar{l}i\bar{j}}
\end{align}

Suppose $X=\sum_i X^i e_i$ and $Y=\sum_j Y^j e_j$ are two nonzero vectors in $T_{1,0}M$, then the pseudo-Hermitian bisectional curvature determined by $X$ and $Y$ is defined by 
\begin{align}
\frac{\sum_{i, j, k, l}R_{i\bar{j}k\bar{l}}X^iX^{\bar{j}}Y^kY^{\bar{l}}}{(\sum_{i}X^iX^{\bar{i}})(\sum_{j}Y^jY^{\bar{j}})}.
\end{align}
If $X=Y$, the above quantity is referred to as the pseudo-Hermitian sectional curvature in the direction $X$ (cf. \cite{W+78}). The pseudo-Hermitian Ricci tensor is defined as 
\begin{align}
R_{i\bar{j}}=\sum_kR_{k\bar{k}i\bar{j}},
\end{align}
and thus the pseudo-Hermitian scalar curvature is given by 
\begin{align}
R=\sum_iR_{i\bar{i}}.
\end{align}

Analogous to Laplace operator in Riemiannian geometry, there is a degenerate elliptic operator in CR geometry which is called sub-Laplace operator. For a $C^2$ function $u:M\rightarrow \mathbb{R}$, $du$ is a smooth section of $T^*M$. Let $\nabla du$ be the covariant derivative of $du\in\Gamma(T^*M)$ with respect to the Tanaka-Webster connection. Therefore, the sub-Laplace operator can be defined by
\begin{align}
\Delta_b u=tr_H(\nabla du)=\sum_i(u_{i\bar{i}}+u_{\bar{i}i}),
\end{align}
where $u_{i\bar{i}}=(\nabla du)(e_i, e_{\bar{i}})$.

In \cite{CDRZ18}, the authors give a sub-Laplacian comparison theorem in pseudo-Hermitian geometry which plays a similar role as the Laplacian comparison theorem in Riemannian geometry.
\begin{lemma}\label{lemma 2.2}
Let $(M^{2m+1},HM,J,\theta)$ be complete pseudo-Hermitian manifold with pseudo-Hermitian Ricci curvature bounded from below by $-k\leq0$ and $\|A\|_{C^1}\leq k_1$ $(k_1\geq 0)$. Let $x_0$ be a fixed point in $M$. Then for any $x\in M$ which is not in the cut locus of $x_0$, there exists $C=C(m)$ such that
\[\Delta_b \gamma(x)\leq C(1/\gamma+\sqrt{1+k+k_1+k_1^2})\]
where $\gamma(x)$ is Riemannian distance of $g_\theta$ between $x_0$ and $x$ in $M$, and $\|A\|_{C^1}=\max_{y\in M}\{\|A\|(y), \|\nabla A\|(y)\}$.
Moreover, $\Delta_b \gamma$ is uniformly bounded from above when $\gamma\geq1$.
\end{lemma}

Let $N$ be a Hermitian manifold of complex dimension $n$. Let $(\eta_\alpha)_{\alpha=1}^n$ be a unitary frame field of $N$ and $(\omega^\alpha)_{\alpha=1}^n$ be its coframe field. Then the Hermitian metric $h$ of $N$ is given by
\begin{align}
h=\sum_\alpha\omega^\alpha\omega^{\bar{\alpha}}.
\end{align}
It is well-known that there are connection 1-forms $(\omega^\alpha_\beta)$ such that
\begin{gather}
d\omega^\alpha=\sum_\beta\omega^\beta\wedge\omega^\alpha_\beta+\Omega^\alpha\label{2.20}\\
\omega^\alpha_\beta+\omega^{\bar{\beta}}_{\bar{\alpha}}=0\label{2.21}
\end{gather}
where
\begin{gather}
\Omega^\alpha=\frac{1}{2}\sum_{\beta, \gamma}T^\alpha_{\beta\gamma}\omega^\beta\wedge\omega^\gamma\label{2.22}\\
T^\alpha_{\beta\gamma}=-T^\alpha_{\gamma\beta}.
\end{gather}
Note that this connection is usually called the Chern connection. The curvature forms $(\Omega^\alpha_\beta)$ are defined by
\begin{gather}
\Omega^\alpha_\beta=d\omega^\alpha_\beta-\sum_\gamma\omega^\gamma_\beta\wedge\omega^\alpha_\gamma\label{2.24}
\end{gather}
and according to \eqref{2.21}, we have
\begin{gather}
\Omega^\alpha_\beta=-\Omega^{\bar{\beta}}_{\bar{\alpha}}=\sum_{\gamma, \delta}R^\alpha_{\beta\gamma\bar{\delta}}\omega^\gamma\wedge\omega^{\bar{\delta}}.
\end{gather}
We set $R_{\alpha\bar{\beta}\gamma\bar{\delta}}=R^\beta_{\alpha\gamma\bar{\delta}}$, then the skew-Hermitian symmetry of $\Omega^\alpha_\beta$ is equivalent to $R_{\alpha\bar{\beta}\gamma\bar{\delta}}=R_{\bar{\beta}\alpha\bar{\delta}\gamma}$. If $Z=\sum_\alpha Z^\alpha \eta_\alpha$ and $W=\sum_{\beta} W^\beta \eta_\beta$ are two tangent vectors in $T_{1,0}N$, then the holomorphic bisectional curvature determined by $Z$ and $W$ is defined by
\begin{align}
\frac{\sum_{\alpha\beta\gamma\delta}R_{\alpha\bar{\beta}\gamma\bar{\delta}}Z^\alpha Z^{\bar{\beta}} W^\gamma W^{\bar{\delta}}}{(\sum_{\alpha}Z^\alpha Z^{\bar{\alpha}})(\sum_{\beta}W^\beta W^{\bar{\beta}})}.
\end{align}
If $Z=W$, the above quantity is called the holomorphic sectional curvature in the direction $Z$. The Ricci tensor is defined as 
\begin{align}
R_{\alpha\bar{\beta}}=\sum_\gamma R_ {\gamma\bar{\gamma}\alpha\bar{\beta}}
\end{align}
and the scalar curvature is given by 
\begin{align}
S=\sum_{\alpha}R_{\alpha\bar{\alpha}}.
\end{align}

For a $C^2$ function $v: N\rightarrow\mathbb{R}$ on the Hermitian manifold $N$, the Laplacian of $v$ is defined as the trace of the Hessian matrix of $v$ with respect to the Chern connection $\nabla^c$. In terms of the local frame fields $(\eta_\alpha)_{\alpha=1}^n$,  it is a well-known fact that
\begin{align}
\Delta v=\sum_{\alpha}(v_{\alpha\bar{\alpha}}+v_{\bar{\alpha}\alpha})=2\sum_\alpha v_{\alpha\bar{\alpha}},
\end{align}
where $ v_{\alpha\bar{\alpha}}=(\nabla^c dv)(\eta_\alpha, \eta_{\bar{\alpha}}).$
\section{ Bochner formulas}
This section will derive the Bochner formulas of generalized holomorphic maps between pseudo-Hermitian manifolds and Hermitian manifolds.

In pseudo-Hermitian geometry, there is an analogue of the holomorphic function on a complex manifold, which is called the CR function. 
\begin{definition}[(cf. \cite{DT07})]
Let $(M,HM,J,\theta)$ be a pseudo-Hermitian manifold. A smooth function $f:M\rightarrow \mathbb{C}$ is said to be a CR function, if $Zf=0$ for all $Z\in T_{0,1}M$. Moreover, $f$ is called a basic CR function if the CR function $f$ satisfies $df(\xi)=0$, where $\xi$ is Reeb vector field.
\end{definition}
We know that a smooth map between two complex manifolds is called holomorphic if its differential commutes with the complex structures at each point. Following this idea, one may define generalized holomorphic maps from pseudo-Hermitian manifolds to Hermitian manifolds as follows.
\begin{definition}[(cf. \cite{GIP01} or \cite{CDRY17})]
\label{definition 3}
Let $(M,HM,J,\theta)$ be a pseudo-Hermitian manifold and $(N,J^N)$ be a complex manifold. A smooth map $f: (M,HM,J,\theta)\rightarrow(N,J^N)$ is called a $(J,J^N)$-holomorphic map if it satisfies
\begin{align}
df\circ J=J^N\circ df.\label{3.1}
\end{align}
\end{definition}
\begin{remark}
\begin{enumerate}[(i)]
\item When the target manifold $N$ is complex plane $\mathbb{C}$ endowed with the canonical complex structure $J^{\mathbb{C}}$. A smooth function $f:M\rightarrow \mathbb{C}$ is a $(J,J^{\mathbb{C}})$-holomorphic map if and only if it is a basic CR function, that is, $df(\xi)=0$ and $Zf=0$ for any $Z\in \Gamma(T_{0,1}M)$.
\item The authors in \cite{CDRY17} discussed Siu-Sampson type theorem for $(J,J^N)$-holomorphic maps.
\end{enumerate}

\end{remark}
Let $f: (M,HM,J,\theta)\rightarrow(N,J^N,h)$ be a smooth map with $\dim_{\mathbb{R}} M=2m+1$ and $\dim_{\mathbb{C}}N=n$. Then in terms of local frames in section \ref{sec2}, we can write $df$ as
\begin{align}
df=\sum_{A, B}f^A_B\theta^B\otimes \eta_A,\label{3.2}
\end{align}
where
\[A=1, 2,\ldots n, \bar{1}, \bar{2},\ldots, \bar{n}\]
\[B=0, 1, 2,\ldots m, \bar{1}, \bar{2},\ldots, \bar{m},\]
\[\theta^0=\theta.\]
Assume $f$ is $(J, J^N )$-holomorphic. Clearly the condition \eqref{3.1} leads $f^\alpha_{\bar{i}}=f^{\bar{\alpha}}_i=f^\alpha_0=f^{\bar{\alpha}}_0=0$ and 
\begin{align}
f^*\omega^\alpha=\sum_if^\alpha_i\theta^i.\label{3.3}
\end{align}

For simplification, denote $\hat{\omega}^\alpha_\beta=f^*\omega^\alpha_\beta, \hat{T}^\alpha_{\beta\gamma}=f^*T^\alpha_{\beta\gamma}, \hat{\Omega}^\alpha_\beta=f^*\Omega^\alpha_\beta$, etc.
Taking the exterior derivative of \eqref{3.3} and using \eqref{2.10}, \eqref{2.20}, \eqref{2.22}, we get
\begin{align}
\sum_iDf^\alpha_i\wedge\theta^i=\frac{1}{2}\sum_{\beta,\gamma,i,j}\hat{T}^\alpha_{\beta\gamma}f^\beta_if^\gamma_j\theta^i\wedge\theta^j-\sum_if^\alpha_i\theta\wedge\tau^i\label{3.4}
\end{align}
where
\begin{align}
Df^\alpha_i=df^\alpha_i+\sum_\beta f^\beta_i\hat{\omega}^\alpha_\beta-\sum_jf^\alpha_j\theta^j_i=f^\alpha_{i0}\theta+\sum_k(f^\alpha_{ik}\theta^k+f^\alpha_{i\bar{k}}\theta^{\bar{k}}).\label{3.5}
\end{align}
From \eqref{3.4}, it follows that
\begin{gather}
f^\alpha_{i0}=f^\alpha_{i\bar{j}}=0\label{3.6}\\
f^\alpha_{ij}=f^\alpha_{ji}-\sum_{\beta,\gamma}\hat{T}^\alpha_{\beta\gamma}f^\beta_if^\gamma_j.\label{3.7}
\end{gather}
By \eqref{3.5} and \eqref{3.6}, we have
\begin{align}
df^\alpha_i=\sum_kf^\alpha_{ik}\theta^k+\sum_jf^\alpha_j\theta^j_i-\sum_\beta f^\beta_i\hat{\omega}^\alpha_\beta.\label{3.8}
\end{align}
Taking the exterior derivative of \eqref{3.8}, and applying structure equations \eqref{2.11} and \eqref{2.24}, one finds that
\begin{align}
\sum_k Df^\alpha_{ik}\wedge\theta^k=\sum_\beta f^\beta_i\hat{\Omega}^\alpha_\beta-\sum_j f^\alpha_j\Pi^j_i-\sum_k f^\alpha_{ik}\theta\wedge\tau^k\label{3.9}
\end{align}
where
\[Df^\alpha_{ik}=df^\alpha_{ik}-\sum_j(f^\alpha_{ij}\theta^j_k-f^\alpha_{jk}\theta^j_i)+\sum_\beta f^\beta_{ik}\hat{\omega}^\alpha_\beta=f^\alpha_{ik0}\theta+\sum_j(f^\alpha_{ikj}\theta^j+f^\alpha_{ik\bar{j}}\theta^{\bar{j}}).\]
Hence,
\begin{align}
f^\alpha_{ik\bar{j}}=\sum_lf^\alpha_lR^l_{ik\bar{j}}-\sum_{\beta,\gamma,\delta}f^\beta_if^\gamma_kf^{\bar{\delta}}_{\bar{j}}\hat{R}^\alpha_{\beta\gamma\bar{\delta}}.\label{3.10}
\end{align}
Set
\begin{align}
u=\sum_{\alpha,i}f^\alpha_if^{\bar{\alpha}}_{\bar{i}}.\label{3.11}
\end{align}
Since \eqref{3.6}, the horizontal differential of $u$ is given by
\begin{align*}
d_Hu=\sum_{\alpha,i,k}(f^\alpha_{ik}f^{\bar{\alpha}}_{\bar{i}}\theta^k+f^{\bar{\alpha}}_{\bar{i}\bar{k}}f^\alpha_i\theta^{\bar{k}}),
\end{align*}
thus,
\[u_k=\sum_{\alpha,i}f^\alpha_{ik}f^{\bar{\alpha}}_{\bar{i}}.\]
By computing $d_Hu_k$, we obtain
\begin{align*}
u_{k\bar{k}}=\sum_{\alpha,i}(f^\alpha_{ik\bar{k}}f^{\bar{\alpha}}_{\bar{i}}+|f^\alpha_{ik}|^2).
\end{align*}
Using \eqref{3.10}, we get
\begin{align*}
u_{k\bar{k}}=\sum_{\alpha,i}|f^\alpha_{ik}|^2+\sum_{\alpha,i,l}f^\alpha_lf^{\bar{\alpha}}_{\bar{i}}R^l_{ik\bar{k}}-\sum_{\alpha,\beta,\gamma,\delta,i}f^{\bar{\alpha}}_{\bar{i}}f^\beta_if^\gamma_kf^{\bar{\delta}}_{\bar{k}}\hat{R}^\alpha_{\beta\gamma\bar{\delta}}.
\end{align*}
Hence, we have the following lemma.
\begin{lemma}
Suppose $f: (M,HM,J,\theta)\rightarrow(N,J^N,h)$ be a $(J,J^N)$-holomorphic map, then we have 
\begin{align}
f^*h\leq u G_\theta,\label{2.12}
\end{align}
and
\begin{align}
\frac{1}{2}\Delta_bu=\sum_{\alpha,i,k}|f^\alpha_{ik}|^2+\sum_{\alpha,i,l}f^\alpha_lf^{\bar{\alpha}}_{\bar{i}}R_{i\bar{l}}-\sum_{\alpha,\beta,\gamma,\delta,i,k}f^{\bar{\alpha}}_{\bar{i}}f^\beta_if^\gamma_kf^{\bar{\delta}}_{\bar{k}}\hat{R}_{\beta\bar{\alpha}\gamma\bar{\delta}}.\label{3.13}
\end{align}
\end{lemma}

In the rest of this section, we turn to generalized holomorphic maps from Hermitian manifolds to pseudo-Hermitian manifolds.
\begin{definition}
\label{definition 4}\label{definition 4}
Let $(M, HM, J, \theta)$ be a pseudo-Hermitian manifold and $(N, J^N)$ be a complex manifold. A  smooth map $g: N\rightarrow M$ is called  a $(J^N, J)$-holomorphic map if it satisfies
\begin{align}
dg_H\circ J^N=J\circ dg_H,\label{3.14}
\end{align}
where $dg_H=\pi_H\circ dg$, $\pi_H: TM\rightarrow HM$ is the natural projection. Moreover,  the $(J^N, J)$-holomorphic map $g$ is said to be horizontally constant if $ dg_H\equiv 0$.
\end{definition}
\begin{remark}
Every $(J^N, J)$-holomorphic map $g: N\rightarrow M$ satisfying $dg(TN)\subseteq HM$, i.e.,  $dg\circ J^N=J\circ dg$, is constant. Indeed, if not, there exists a local vector field $e\in T^{1,0}N$ such that $dg(e)\neq 0$, so $HM \ni dg[e, \bar{e}]=[dg(e), dg(\bar{e})]\notin HM$, which leads to a contradiction.
\end{remark}
Suppose $(N, J^N )$ is a complex manifold and $(M, HM, J, \theta)$ is a pseudo-Hermitian manifold. Let $g : N\rightarrow M$ be a $(J^N , J )$-holomorphic map. Under the local frames in Section \ref{sec2}, we can express the differential of $g$ as 
\begin{align}
dg=\sum_{A,B}g^B_A\omega^A\otimes e_B,
\end{align}
where $e_0=\xi$ and the values of $A, B$ are the same as those of \eqref{3.2}. Clearly, the condition  \eqref{3.14} in Definition \ref{definition 4} is equivalent to
\begin{align}
g^i_{\bar{\alpha}}=g^{\bar{i}}_\alpha=0.\label{3.16}
\end{align}
From \eqref{3.14} and \eqref{3.16}, we have
\begin{gather}
g^*\theta^i=\sum_{\alpha} g^i_\alpha\omega^\alpha\label{3.17}\\
g^*\theta=\sum_{\alpha}(g^0_\alpha\omega^\alpha+g^0_{\bar{\alpha}}\omega^{\bar{\alpha}}).
\end{gather}

To simplify the notations, we set $\hat{\theta}^i_j=g^*\theta^i_j, \hat{A}^i_j=g^*A^i_j, \hat{R}^i_{j\bar{k}l}=R^i_{j\bar{k}l}$, etc. By taking the exterior derivative of \eqref{3.17} and using the structure equations in $M$ and $N$,  we get
\begin{align}
\sum_{\alpha}Dg^i_\alpha\wedge\omega^\alpha+\sum_\alpha g^i_\alpha\Omega^\alpha-\sum_{\alpha,\beta, j}g^0_\alpha g^{\bar{j}}_{\bar{\beta}}\hat{A}^i_{\bar{j}}\omega^\alpha\wedge\omega^{\bar{\beta}}-\sum_{\alpha,\beta, j}g^0_{\bar{\alpha}}g^{\bar{j}}_{\bar{\beta}}\hat{A}^i_{\bar{j}}\omega^{\bar{\alpha}}\wedge\omega^{\bar{\beta}}=0,\label{3.19}
\end{align}
where
\begin{align}
Dg^i_\alpha=dg^i_\alpha-\sum_\gamma g^i_\gamma\omega^\gamma_\alpha+\sum_jg^j_\alpha\hat{\theta}^i_j=\sum_\beta (g^i_{\alpha\beta}\omega^{\beta}+ g^i_{\alpha\bar{\beta}}\omega^{\bar{\beta}}).\label{3.20}
\end{align}
Then \eqref{3.19} gives 
\begin{gather}
g^i_{\alpha\beta}=g^i_{\beta\alpha}-\sum_\gamma g^i_\gamma T^\gamma_{\beta\alpha}\notag\\
g^i_{\alpha\bar{\beta}}=-\sum_j g^0_\alpha g^{\bar{j}}_{\bar{\beta}}\hat{A}^i_{\bar{j}}\label{3.21}\\
\sum_j (g^0_{\bar{\alpha}}g^{\bar{j}}_{\bar{\beta}}-g^0_{\bar{\beta}}g^{\bar{j}}_{\bar{\alpha}})\hat{A}^i_{\bar{j}}=0\notag
\end{gather}

Taking the exterior derivative of \eqref{3.20} and using the structure equations again, we obtain
\begin{align}
\sum_\beta Dg^i_{\alpha\beta}\wedge\omega^{\beta}+\sum_\beta Dg^i_{\alpha\bar{\beta}}\wedge\omega^{\bar{\beta}}=-\sum_\gamma g^i_{\gamma}\Omega^\gamma_\alpha+\sum_j g^j_\alpha g^*\Pi^i_j,\label{3.22}
\end{align}
where
\begin{gather}
Dg^i_{\alpha\beta}=dg^i_{\alpha\beta}-\sum_\gamma (g^i_{\alpha\gamma}\omega^\gamma_\beta+g^i_{\gamma\beta}\omega^\gamma_\alpha)+\sum_j g^j_{\alpha\beta}\hat{\theta}^i_j=\sum_\gamma (g^i_{\alpha\beta\gamma}\omega^\gamma+g^i_{\alpha\beta\bar{\gamma}}\omega^{\bar{\gamma}})\\
Dg^i_{\alpha\bar{\beta}}=dg^i_{\alpha\bar{\beta}}-\sum_\gamma (g^i_{\alpha\bar{\gamma}}\omega^{\bar{\gamma}}_{\bar{\beta}}+g^i_{\gamma\bar{\beta}}\omega^\gamma_\alpha)+\sum_j g^j_{\alpha\bar{\beta}}\hat{\theta}^i_j=\sum_\gamma(g^i_{\alpha\bar{\beta}\gamma}\omega^\gamma+g^i_{\alpha\bar{\beta}\bar{\gamma}}\omega^{\bar{\gamma}}).
\end{gather}
Compare the $(1,1)$-form of  \eqref{3.22} and get
\begin{align}
g^i_{\alpha\beta\bar{\gamma}}=g^i_{\alpha\bar{\gamma}\beta}+\sum_\delta g^i_\delta R^\delta_{\alpha\beta\bar{\gamma}}-\sum_{j,k,l}g^j_\alpha g^k_\beta g^{\bar{l}}_{\bar{\gamma}}\hat{R}^i_{jk\bar{l}}-\sum_{j,k}(g^j_\alpha g^k_\beta g^0_{\bar{\gamma}}\hat{W}^i_{jk}+g^j_\alpha g^0_\beta g^{\bar{k}}_{\bar{\gamma}}\hat{W}^i_{j\bar{k}}).\label{3.25}
\end{align}
Set 
\begin{align}
v=\sum_{i, \alpha}g^i_\alpha g^{\bar{i}}_{\bar{\alpha}}.\label{3.26}
\end{align}
By direct computation, we have
\begin{align}
\frac{1}{2}\Delta v=\sum_{i,\alpha,\gamma}(|g^i_{\alpha\gamma}|^2+|g^i_{\alpha\bar{\gamma}}|^2)+\sum_{i,\alpha,\gamma}(g^{\bar{i}}_{\bar{\alpha}}g^i_{\alpha\gamma\bar{\gamma}}+g^i_\alpha g^{\bar{i}}_{\bar{\alpha\gamma\bar{\gamma}}}).\label{3.28}
\end{align}
Using \eqref{3.21} and \eqref{3.25}, we perform the following computation
\begin{align}
g^i_{\alpha\gamma\bar{\gamma}}&=g^i_{\gamma\alpha{\bar{\gamma}}}-(\sum_\beta g^i_{\beta}T^\beta_{\gamma\alpha})_{\bar{\gamma}}\notag\\
						&=g^i_{\gamma\bar{\gamma}\alpha}+\sum_\delta g^i_\delta R^\delta_{\gamma\alpha\bar{\gamma}}-\sum_{j,k,l}g^j_\gamma g^k_\alpha g^{\bar{l}}_{\bar{\gamma}}\hat{R}^i_{jk\bar{l}}-(\sum_\beta g^i_\beta T^\beta_{\gamma\alpha})_{\bar{\gamma}}\notag\\
						&\  \  \   -\sum_{j,k}(g^j_\gamma g^k_\alpha g^0_{\bar{\gamma}}\hat{W}^i_{jk}+g^j_\gamma g^0_\alpha g^{\bar{k}}_{\bar{\gamma}}\hat{W}^i_{j\bar{k}})\label{3.29}
\end{align}
and
\begin{align}
g^{\bar{i}}_{\bar{\alpha}\gamma\bar{\gamma}}=-(\sum_j g^0_{\bar{\alpha}}g^j_\gamma \hat{A}^{\bar{i}}_j)_{\bar{\gamma}}.\label{3.30}
\end{align}
 From \eqref{3.21}, \eqref{3.28}, \eqref{3.29}, \eqref{3.30}, we obtain the Bochner formula as follows.
 \begin{lemma}
 Suppose $f: (N, J^N, h)\rightarrow (M, HM, J, \theta)$ is a $(J^N,J)$-holomorphic map, then we have
 \begin{align}
  g^*G_\theta \leq vh, \label{3.30'}
 \end{align}
 and
\begin{align}
\frac{1}{2}\Delta v=\sum_{i,\alpha,\gamma}&(|g^i_{\alpha\gamma}|^2+|g^i_{\alpha\bar{\gamma}}|^2)+\sum_{i,\alpha,\delta}g^{\bar{i}}_{\bar{\alpha}}g^i_\delta R_{\alpha\bar{\delta}}-\sum_{i,j,k,l,\alpha,\gamma}g^{\bar{i}}_{\bar{\alpha}}g^j_\gamma g^k_\alpha g^{\bar{l}}_{\bar{\gamma}}\hat{R}^i_{jk\bar{l}}\notag\\
			-&\sum_{i,j,k,\alpha,\gamma}(g^{\bar{i}}_{\bar{\alpha}}g^j_\gamma g^k_\alpha g^0_{\bar{\gamma}}\hat{W}^i_{jk}+g^{\bar{i}}_{\bar{\alpha}}g^j_\gamma g^0_\alpha g^{\bar{k}}_{\bar{\gamma}}\hat{W}^i_{j\bar{k}})\notag\\
			-\sum_{i,\alpha,\gamma}&\{g^{\bar{i}}_{\bar{\alpha}}(\sum_j g^0_\gamma g^{\bar{j}}_{\bar{\gamma}}\hat{A}^i_{\bar{j}})_\alpha+g^i_\alpha(\sum_j g^0_{\bar{\alpha}}g^j_\gamma \hat{A}^{\bar{i}}_j)_{\bar{\gamma}}\}-\sum_{i,\alpha,\gamma}g^{\bar{i}}_{\bar{\alpha}}(\sum_\beta g^i_\beta T^\beta_{\gamma\alpha})_{\bar{\gamma}}\label{3.31}
\end{align}
\end{lemma}

\section{Schwarz type lemma for $(J,J^N)$-holomorphic maps}\label{sec4}
In this section, we will establish the Schwarz type lemma for $(J,J^N)$-holomorphic maps. As the corollaries of this lemma, the Liouville theorem and little Picard theorem for basic CR functions are given.

Suppose that $(M,HM,J,\theta)$ is a complete pseudo-Hermitian manifold and $(N,J^N,h)$ is a Hermitian manifold. Let $f: M\rightarrow N$ be a $(J,J^N)$-holomorphic map. To give the Schwarz type lemma for $f$, it is sufficient to estimate the upper bound of $u$ defined by \eqref{3.11} due to \eqref{2.12}. Define
\begin{align}
\phi(x)=(a^2-\gamma^2(x))^2u
\end{align}
where $\gamma(x)$ is the Riemannian distance from a fixed point $z$ to $x$ in $M$. Let $B_a(z)$ be a open geodesic ball in $M$ with its center at $z$ and radius of $a$. It is obvious that $\phi(x)$ attains its maximum in $B_a(z)$. Suppose $x_0$ is a maximum point. For maximum principle, $\gamma$ is required to be twice differentiable near $x_0$. This may be remedied by the following consideration(cf. \cite{CCL79}): Let $\tau:[0, \gamma(x_0)]\rightarrow M$ be a minimizing geodesic joining $z$ and $x_0$ such that $\tau(0)=z$ and $\tau(\gamma(x_0))=x_0$. If $x_0$ is a cut point of $z$, then for a small number $\varepsilon>0$, $x_0$ is not a conjugate point of $\tau(\varepsilon)$ along $\tau$. It is well-known that there is a cone $\mathfrak{C}$ with its vertex at $\tau(\varepsilon)$ and containing a neighborhood of $x_0$. Let $\gamma_\varepsilon(x)$ denote the Riemannian distance from $\tau(\varepsilon)$ to $x$, then $\gamma_\varepsilon$ is smooth near $x_0$. Let $\tilde{\gamma}(x)=\varepsilon+\gamma$, then we have $\gamma\leq\tilde{\gamma}$ and the equality holds at $x_0$. So we can consider the function $(a^2-\tilde{\gamma}^2)^2u$, it also attains the maximum at $x_0$. Let $\varepsilon\rightarrow 0$, we may assume that $\gamma$ is smooth near $x_0$. Therefore, applying the maximum principle to $\phi$, at $x_0$, we have
\begin{gather}
\frac{\nabla^H u}{u}=-2\frac{\nabla^H(a^2-\gamma^2)}{a^2-\gamma^2}\label{4.2}\\
\frac{\Delta_bu}{u}+4\frac{\nabla^H(a^2-\gamma^2)}{(a^2-\gamma^2)}\cdot\frac{\nabla^Hu}{u}+2\frac{\Delta_b(a^2-\gamma^2)}{a^2-\gamma^2}+2\frac{\|\nabla^H(a^2-\gamma^2)\|^2}{(a^2-\gamma^2)^2}\leq0\label{4.3}
\end{gather}
where the inner product $\cdot$ and the norm $\|\cdot\|$ is induced by the Webster metric $g_\theta$.
Substituting \eqref{4.2} into \eqref{4.3}, we obtain
\begin{align}
\frac{\Delta_bu}{u}-6\frac{\|\nabla^H(a^2-\gamma^2)\|^2}{(a^2-\gamma^2)^2}-2\frac{\Delta_b\gamma^2}{a^2-\gamma^2}\leq0.\label{4.4}
\end{align}
Let $-K_1$ be the greatest lower bound of the pseudo-Hermitian Ricci curvature of $M$, and $-K_2$ be the least upper bound of the holomorphic bisectional curvature of $N$. Then it follows from \eqref{3.13} that
\begin{align}
\frac{1}{2}\Delta_bu\geq-K_1u+K_2u^2.\label{4.5}
\end{align}
Using $\|\nabla^H\gamma\|\leq1$, we get
\begin{align}
\|\nabla^H(a^2-\gamma^2)\|^2=\|2\gamma\nabla^H\gamma\|^2\leq 4a^2.\label{4.6}
\end{align}
If $\|A\|_{C^1}$ is bounded from above on $M$, then by Lemma \ref{lemma 2.2}, we have
\begin{align}
\Delta_b\gamma^2=2\gamma\Delta_b\gamma+2\|\nabla^H\gamma\|^2\leq C(1+a)\label{4.7}
\end{align}
where $C$ is a positive constant independent of $a$. Suppose that $K_1\geq0$ and $K_2>0$. Then substituting \eqref{4.5}, \eqref{4.6} and \eqref{4.7} into \eqref{4.4}, we obtain
\[u(x_0)\leq\frac{K_1}{K_2}+\frac{12a^2}{K_2(a^2-\gamma^2(x_0))^2}+\frac{C(1+a)}{K_2(a^2-\gamma^2(x_0))}. \]
Thus,
\begin{align*}
(a^2-\gamma^2(x))^2u(x)&\leq(a^2-\gamma^2(x_0))^2u(x_0)\\
                      &\leq\frac{K_1}{K_2}(a^2-\gamma^2(x_0))^2+\frac{12a^2}{K_2}+\frac{C(1+a)}{K_2}(a^2-\gamma^2(x_0))\\
                      &\leq\frac{K_1}{K_2}a^4+\frac{12a^2}{K_2}+\frac{C(1+a)}{K_2}a^2
\end{align*}
for any $x\in B_a(z)$.
It follows that
\[u(x)\leq\frac{K_1a^4}{K_2(a^2-\gamma^2(x))^2}+\frac{12a^2}{K_2(a^2-\gamma^2(x))^2}+\frac{C(1+a)a^2}{K_2(a^2-\gamma^2(x))^2}.\]
Let $a\rightarrow\infty$, we deduce that
\[\sup_M u\leq \frac{K_1}{K_2}.\]
Due to \eqref{2.12}, we obtain
\[f^*h\leq\frac{K_1}{K_2}G_\theta.\]

Therefore, we have the Schwarz type lemma as follows:
\begin{theorem}\label{theorem 4.1}
Let $(M^{2m+1},HM,J,\theta)$ be a complete pseudo-Hermitian manifold with pseudo-Hermitian Ricci curvature bounded from below by $-K_1\leq0$ and $\|A\|_{C^1}$ bounded from above. Let $(N^n,J^N,h)$ be a Hermitian manifold with holomorphic bisectional curvature bounded from above by $-K_2<0$. Then for any $(J,J^N)$-holomorphic map $f: M\rightarrow N$, we have
\begin{align}
f^*h\leq\frac{K_1}{K_2}G_\theta.\label{4.8}
\end{align}
In particular, if $K_1=0$, every $(J,J^N)$-holomorphic map from $M$ into $N$ is constant.
\end{theorem}
\begin{remark}\label{rk3}
Using Royden's lemma (cf. \cite{Roy80}), we can weaken the hypothesis on $N$ by assuming the holomorphic sectional curvature of $N$ is bounded above by $-K_2<0$, but the constant $\frac{K_1}{K_2}$ in \eqref{4.8} will be replaced by $\frac{2v}{v+1}\frac{K_1}{K_2}$, where $v$ is the maximal rank of $df$.
\end{remark}
Since a unit disk in complex plane equipped with Poincar\'{e} metric is a K\"{a}hler manifold with constant negative holomorphic curvature, we have the following Liouville theorem:
\begin{corollary}\label{corollary 4.2}
Let $(M,HM,J,\theta)$ be a complete pseudo-Hermitian manifold with non-negative pseudo-Hermitian Ricci curvature and $\|A\|_{C^1}$ bounded from above. Any bounded basic CR function on $M$ is constant.
\end{corollary}
Using the fact that the complex plane $\mathbb{C}$ minus two distinct points admits a complete Hermitian metric with holomorphic curvature less than a negative constant \cite{Kob05}, we derive little Picard theorem for basic CR functions.
\begin{corollary}\label{corollary 4.3}
Let $(M,HM,J,\theta)$ be a complete pseudo-Hermitian manifold with non-negative pseudo-Hermitian Ricci curvature and $\|A\|_{C^1}$ bounded from above. Any basic CR function $u: M\rightarrow \mathbb{C}$ missing more than one point in its image is constant.
\end{corollary}

\section{Schwarz type lemma for $(J^N, J)$-holomorphic maps}
In this section, we will establish the Schwarz type lemma for  $(J^N, J)$-holomorphic maps. In \cite{Yau78}, Yau have proved that
\begin{proposition}\label{proposition 5.1}
Let $N$ be a complete K\"ahler manifold with Ricci curvature bounded below. A non-negative smooth function $u$ on $N$ satisfies the following inequality
\[\Delta u\geq -k_1u+ k_2u^2,\]
where $k_1\geq 0, k_2>0.$
Then $\sup_M u \leq \frac{k_1}{k_2}.$
\end{proposition}
It is notable that Tosatti has generalized this proposition to almost Hermitian manifold in \cite{Tos07}.

Let $g:(N^n,J^N,h)\rightarrow (M^{2m+1},HM,J,\theta) $ be a $(J^N, J)$-holomorphic map. If $N$ is a complete K\"ahler manifold and $M$ is a Sasakian manifold, by \eqref{3.31}, we obtain 
\begin{align}
\frac{1}{2}\Delta v=\sum_{i,\alpha,\gamma}(|g^i_{\alpha\gamma}&|^2+|g^i_{\alpha\bar{\gamma}}|^2)+\sum_{i,\alpha,\delta}g^{\bar{i}}_{\bar{\alpha}}g^i_\delta R_{\alpha\bar{\delta}}-\sum_{\alpha,\gamma,i,j,k,l}g^{\bar{i}}_{\bar{\alpha}}g^j_\gamma g^k_\alpha g^{\bar{l}}_{\bar{\gamma}}\hat{R}_{j\bar{i}k\bar{l}}\label{61}
\end{align}
where $v$ is defined by \eqref{3.26}.
Let $-K_1$ be the greatest lower bound of the Ricci curvature of $N$, and $-K_2$ be the least upper bound of the pseudo-Hermitian bisectional curvature of $M$, where $K_1\geq 0, K_2>0$. It follows from \eqref{61} that 
\begin{align}
\frac{1}{2}\Delta v\geq -K_1v+K_2v^2.
\end{align}
By Proposition \ref{proposition 5.1}, we deduce that $\sup v\leq \frac{K_1}{K_2}$, which, combining with \eqref{3.30'}, yields the Schwarz type lemma for $(J^N, J)$-holomorphic maps as follows.
\begin{theorem}
 Let $(N^n,J^N,h)$ be a complete K\"ahler manifold with Ricci curvature bounded from below by $-K_1\leq0$. Let $(M^{2m+1},HM,J,\theta)$ be a Sasakian manifold with pseudo-Hermitian bisectional curvature bounded from above by $-K_2<0$.  Then for any $(J^N, J)$-holomorphic map $g: N\rightarrow M$, we have
 \begin{align}
 g^*G_\theta \leq\frac{K_1}{K_2}h.\label{5.3}
 \end{align}
 In particular, if $K_1=0$, any $(J^N, J)$-holomorphic map is horizontally constant.
 \begin{remark} 
By Royden's lemma (cf. \cite{Roy80}), we can weaken the hypothesis on $M$ by assuming the pseudo-Hermitian sectional curvature of $M$ is bounded above by $-K_2<0$, but the constant $\frac{K_1}{K_2}$ in \eqref{5.3} will be replaced by $\frac{2v}{v+1}\frac{K_1}{K_2}$, where $v$ is the maximal rank of $dg_H$.
 \end{remark}
\end{theorem}
If $\dim_{\mathbb{C}} N=1$, one can also weaken the hypothesis on $N$.
\begin{corollary}
\label{corollary 5.3}
 Let $(N,J^N,h)$ be a complete 1-dimensional K\"ahler manifold with Ricci curvature bounded from below by $-K_1\leq0$. Let $(M^{2m+1},HM,J,\theta)$ be a Sasakian manifold with pseudo-Hermitian sectional curvature bounded from above by $-K_2<0$. Then for any $(J^N, J)$-holomorphic map $g: N\rightarrow M$, \eqref{5.3} holds.
\end{corollary}

\section{Invariant pseudodistance on CR manifolds}
In this section, we will give an invariant pseudodistance on pseudo-Hermitian manifolds, which are analogous to the Carath\'eodory pseudodistance on complex manifolds(cf. \cite{Ca26}). Note that this notion holds true for more general CR manifolds. 

Let $(M^{2m+1},HM,J,\theta)$ be a pseudo-Hermitian manifold. Let $D$ denote the unit disk in the complex plane and $\rho$ denote the Bergman distance of $D$.  Analogous to the Carath\'eodory pseudodistance on complex manifolds, we may also define CR Carath\'eodory pseudodistance on $M$ by
\begin{align}
c_M(p, q)= \sup_f \rho(f(p), f(q))
\end{align}
where the supremum is taken for all possible CR functions $f: M\rightarrow D$. Note that $f: M\rightarrow D$ is a CR function if and only if it satisfies $df\circ J=J^D\circ df$ on $HM$. It is easy to verify the following axioms for the pseudodiatance: 
\begin{align}
c_M(p, q)\geq 0, \ c_M(p, q)=c_M(q, p), \ c_M(p, r)+c_M(r, q)\geq c_M(p, q).
\end{align}
The most important property of $c_M$ is given as follows, whose proof is trivial.
\begin{proposition}
Let $M, \tilde{M}$ be two pseudo-Hermitian manifolds and let $f:M\rightarrow \tilde{M}$ be a CR map. Then 
\[c_{\tilde{M}}(f(p), f(q))\leq c_M(p, q) \]
\end{proposition}
\begin{corollary}
Let $f:M\rightarrow \tilde{M}$ be a CR isomorphism, then 
\[c_{\tilde{M}}(f(p), f(q))=c_M(p, q) \]
\end{corollary}
In order to apply the  Schwarz lemma in section \ref{sec4}, we may define the basic pseudodistance by
\begin{align}
c'_M(p, q)= \sup_f \rho(f(p), f(q))
\end{align}
where the supremum is taken for all $(J, J^D)$-holomorphic maps $f: M\rightarrow D$. It's clear that
$c'_M\leq c_M$. 
\begin{example}
Let $D$ be the unit disc in $\mathbb{C}$. Set $\theta=dt+i(\bar{\partial}-\partial)\log(1-|z|^2)^{-1}$, $\xi=\frac{\partial}{\partial t}$, and $H=\ker \theta$. We define an almost complex structure $J$ on $H$ to be the horizontal lift of the complex structure $J^D$ on $D$. Then $D^3(-1)= (D\times \mathbb{R}, H, J, \theta)$ is the 3-dimensional Sasakian space form with pseudo-Hermitian sectional curvature $-1$.  (cf. \cite{BG08} or \cite{DRY19}). Choose $T=\frac{\partial}{\partial z}+i\frac{\bar{z}}{1-|z|^2}\frac{\partial}{\partial t}$ as the frame field of $T_{1,0}D^3(-1)$.  It is easy to see that $f=f(z,t): D^3(-1)\rightarrow D$ is a basic CR function if and only if $\frac{\partial f}{\partial t}=\frac{\partial f}{\partial \bar{z}}=0$. Thus, for any $(z_1,t_1), (z_2,t_2)\in D^3(-1)$, $c'_{D^3(-1)}((z_1,t_1), (z_2,t_2))=c'_{D^3(-1)}((z_1,0), (z_2,0))=c^D(z_1, z_2)$, where $c^D$ is the Carath\'eodory distance of unit disk $D$. For $c_{D^3(-1)}$, since $f_1(z,t)=z$ and $f_2(z,t)=\frac{t-i(\log{(1-|z|^2)}+1)}{t-i(\log(1-|z|^2)-1)}$ are CR functions from $D^3(-1)$ to $D$, $c_{D^3(-1)}((z_1,t_1), (z_2,t_2))\geq \rho(z_1, z_2)>0$ for any $z_1\neq z_2$ in $D$, and $c_{D^3(-1)}((z,t_1), (z,t_2))\geq \rho(f_2(z,t_1), f_2(z,t_2))>0$ for $z\in D, t_1\neq t_2$. Therefore, $c_{D^3(-1)}((z_1,t_1), (z_2,t_2))=0$ if and only if $(z_1,t_1)=(z_2,t_2)$. Consequently, for two distinct points $p$ and $q$ in $D^3(-1)$, there is a bounded CR function $f$ on $D^3(-1)$ such that $f(p)\neq f(q)$.
\end{example}

Making use of Theorem \ref{theorem 4.1}, we have 
\begin{theorem}\label{theorem 6.3}
Let $(M,HM,J,\theta)$ be a complete pseudo-Hermitian manifold with pseudo-Hermitian Ricci curvature bounded from below by a constant $-K\leq 0$ and $\|A\|_{C^1}$ bounded from above. Then 
\[c'_M(p,q)\leq \sqrt{K}d^M_{cc}(p, q)<\infty\]
for $p,q\in M$. In particular, if $K=0$, $c'_M\equiv0$.
\end{theorem}
\begin{proof}
Let $ds^2_D$ denote the Poincar\'e metric on $D$, and the curvature of $(D, ds^2_D)$ is -1. For any $(J, J^D)$-holomorphic map $f: M\rightarrow D$, by Theorem \ref{theorem 4.1}, we have 
\begin{align}
f^*ds^2_D\leq KG_\theta.\label{6.4}
\end{align}  
Assume that $p, q$ are two points in $M$ and $\tau: [0, 1]\rightarrow M$ is a horizontal Lipschitz curve between them. Then,
\begin{align}
\rho(f(p), f(q))\leq\int_0^1\sqrt{ds^2_D(f_*(\tau'), f_*(\tau'))}\, dt\leq\sqrt{K}\int_0^1\sqrt{L_\theta(\tau', \tau')}\, dt,
\end{align}
where the second inequality follows from \eqref{6.4}. Taking the supremum with respect to $f$ and infimum with respect to $\tau$, the theorem follows.
\end{proof}

By the Definitions of the Carath\'eodory and CR Carath\'eodory pseudodistances, it is easy to see the relationships between them.
\begin{proposition}\label{proposition 6.4}
Let $(M^{2m+1},HM,J,\theta)$ be a pseudo-Hermitian manifold and $(N, J^N)$ a Hermitian manifold.
\begin{enumerate}[(i)]
\item For any $(J, J^N)$-holomorphic map $f: M\rightarrow N$, we have
\begin{align}
c^N(f(p), f(q))\leq c'_M(p, q)\leq c_M(p, q)
\end{align}
for any $p, q\in M$, where $c^N$ is the Carath\'eodory pseudodistance on $N$.\label{proposition 6.4 (1)}
\item For any $(J^N, J)$-holomorphic map $g:N\rightarrow M$, we have
\begin{align}
c'_M(g(x), g(y))\leq c^N(x, y) 
\end{align}
for any $x, y\in N$.
\end{enumerate}
\end{proposition}
Combining Theorem \ref{theorem 6.3} and Proposition \ref{proposition 6.4} \eqref{proposition 6.4 (1)}, we derive anthor Liouville theorem for $(J, J^N)$-holomorphic maps.
\begin{theorem}\label{theorem 6.5}
Let $(M,HM,J,\theta)$ be a complete pseudo-Hermitian manifold with non-negative pseudo-Hermitian Ricci curvature and $\|A\|_{C^1}$ bounded from above. Let $(N, J^N)$ be a complex manifold whose Carath\'eodory pseudodistance is a distance. Then any $(J, J^N)$-holomorphic map $f: M\rightarrow N$ is constant.
\end{theorem}
\begin{remark}
We can also deduce Corollary \ref{corollary 4.2} from Theorem \ref{theorem 6.5}, since the Carath\'eodory pseudodistance of unit disc $D$ is a distance.
\end{remark}

 \bigskip
   Tian Chong\ \
   
   School of Science, College of Arts and Sciences
   
   Shanghai Polytechnic University
   
   Shanghai, 201209, P. R. China
   
   chongtian@sspu.edu.cn\ \ \
   
\bigskip

   Yuxin Dong
   
   School of Mathematical Sciences
   
   Fudan University
   
   Shanghai, 200433, P. R. China
  
    yxdong@fudan.edu.cn
    
\bigskip

   Yibin Ren
   
   College of Mathematics and Computer Science
   
   Zhejiang Normal University
   
   Jinhua, 321004, Zhejiang, P.R. China
   
   allenryb@outlook.com

\bigskip

   Weike Yu
   
   School of Mathematical Sciences
   
   Fudan University
   
   Shanghai, 200433, P. R. China
   
   wkyu2018@outlook.com


\begin{thebibliography}{99}
\bibitem{Ahl38}
{L.V. Ahlfors}, `An extension of Schwarz's lemma', {\em Trans.
Amer. Math. Soc.}43 (1938), no. 3, 359-364.
%
\bibitem{BG08} 
{ C. Boyer \and K. Galicki}, `Sasakian geometry', {\em Oxford Univ. Press}, 2008.
%
\bibitem{Ca26}
{ C. Carath\'eodery }, `\"Uber das Schwarzsche lemma bei analytischen funktionen von zwei komplexen ver\"anderlichen', {\em Math. Ann.} 97(1926), 76-98.
\bibitem{CCL79}
{ Z. Chen, S.Y. Cheng, \and Q. Lu}, `On the Schwarz lemma for complete k\"ahler manifolds', {\em Scientia Sinica}, XXII (1979), no.11, 1238–1247.
%
\bibitem{CDRY17}
{ T. Chong, Y. Dong, Y. Ren, \and G. Yang}, `On harmonic and pseudoharmonic maps from pseudo-Hermitian manifolds', {\em Nagoya Mathematical Journal} (2017) 1–41.
%
\bibitem{CDRZ18}
{ T. Chong, Y. Dong, Y. Ren, \and W. Zhang}, `Pseudoharmonic maps from complete noncompact pseudo-Hermitian manifolds to regular balls', {\em The Journal of Geometric Analysis} (2018) 1–30.
\bibitem{Che68}
{ S.S. Chern}, `On holomorphic mappings of Hermitian manifolds of the same dimension',{\em In Proc. Symp. Pure Math.} 11 (1968) 157–170.
\bibitem{Cho02}
{ W.L. Chow}. `{\"U}ber Systeme von linearen partiellen Differentialgleichungen erster Ordnung', {\em Math. Ann.}, 117 (1939) 98-105.
\bibitem{DRY19} 
{ Y. Dong, Y. Ren, \and W. Yu}, `Schwarz type lemmas for pseudo-Hermitian manifolds', arXiv:1909.02936.
\bibitem{DT07}
{ S. Dragomir \and G.Tomassini}, `Differential geometry and analysis on CR manifolds', {\em Springer Science \& Business Media}, vol.246, 2007.
\bibitem{GH77}
{ Samuel I Goldberg, Zvi Har}, `A general Schwarz lemma for Riemannian-manifolds', {\em Bull. Greek Math. Soc.} 18 (1977), no.18A, 141–148.
\bibitem{GIP01}
{ C. Gherghe, S. Ianus, \and A. M. Pastore}, `CR manifolds, harmonic maps and stability', {\em Journal of Geometry} 71 (2001), no. 1, 42–53.
\bibitem{Kob05}
{ S. Kobayashi}, `Hyperbolic manifolds and holomorphic mappings: an introduction', {\em World Scientific}, 2005.
\bibitem{Lu68}
{ Y.C. Lu}, `Holomorphic mappings of complex manifolds', {\em J. Differential Geo.} 2 (1968), no.3, 299–312.
\bibitem{Pic15}
{ G. Pick}, `{\"U}ber eine Eigenschaft der konformen Abbildung kreisf{\"o}rmiger Bereiche', {\em Math. Ann.} 77 (1915), no.1, 1–6.
\bibitem{Ras38}
{ P.K. Rashevsky}, `Any two points of a totally nonholonomic space may be connected by an admissible line', {\em Uch. Zap. Ped. Inst. im. Liebknechta, Ser. Phys. Math}, 2 (1938), 83–94.
\bibitem{Roy80}
{ H.L. Royden}, `The Ahlfors-Schwarz lemma in several complex variables', {\em Commentarii Mathematici Helvetici}, 55 (1980) no. 1, 547–558.
\bibitem{She84}
{ C.L. Shen}, `A generalization of the Schwarz-Ahlfors lemma to the theory of harmonic maps', {\em J. reine angew. Math}, 348(1984), 23–33.
\bibitem{Tos07}
{ V. Tosatti}, `A general Schwarz lemma for almost-Hermitian manifolds', {\em Comm. anal. geo.} 15 (2007), no. 5, 1063-1086.
\bibitem{W+78}
{ S.M. Webster}, `Pseudo-Hermitian structures on a real hypersurface', {\em J. Differential Geo.} 13(1978), no. 1, 25–41.
\bibitem{Yau78}
{ S.T. Yau}, `A general Schwarz lemma for K\"ahler manifolds', {\em Ameri. J. Math.}, 100(1978), no. 1, 197–203.
\end{thebibliography}
\end{document}